\title{Hyperfinite graphings and combinatorial optimization}
\author{L\'aszl\'o Lov\'asz\\
Hungarian Academy of Sciences and E\"otv\"os Lor\'and University\\
Budapest}
\date{
\it Dedicated to Endre Szemer\'edi at the occasion of his $80^{\rm th}$
birthday}

\documentclass[11pt,fleqn]{article}
\usepackage{amsmath,amssymb,graphicx,bbm,bm,delarray,pict2e,ntheorem}
\usepackage[hypertex]{hyperref}
\usepackage[normalem]{ulem}
\long\def\ignore#1{}

\begin{document}

\newtheorem{theorem}{Theorem}
\newtheorem{prop}[theorem]{Proposition}
\newtheorem{lemma}[theorem]{Lemma}
\newtheorem{claim}{Claim}
\newtheorem{corollary}[theorem]{Corollary}
\theorembodyfont{\rmfamily}
\newtheorem{remark}[theorem]{Remark}
\newtheorem{example}{Example}
\newtheorem{conj}{Conjecture}
\newtheorem{problem}[theorem]{Problem}
\newtheorem{step}{Step}
\newtheorem{alg}{Algorithm}
\newenvironment{proof}{\medskip\noindent{\bf Proof. }}{\hfill$\square$\medskip}

\def\R{\mathbb{R}}
\def\one{\mathbbm1}
\def\T{^{\sf T}}
\def\Pr{{\sf P}}
\def\E{{\sf E}}
\def\Q{{\mathbf Q}}
\def\bd{\text{bd}}
\def\eps{\varepsilon}
\def\wh{\widehat}
\def\cork{\text{\rm corank}}
\def\rank{\text{\rm rank}}
\def\Ker{\text{\rm Ker}}
\def\rk{\text{\rm rank}}
\def\supp{\text{\rm supp}}
\def\diag{\text{\rm diag}}
\def\sep{\text{\rm sep}}
\def\tr{\text{\rm tr}}
\def\iso{{h}}

\def\AA{\mathcal{A}}\def\BB{\mathcal{B}}\def\CC{\mathcal{C}}
\def\DD{\mathcal{D}}\def\EE{\mathcal{E}}\def\FF{\mathcal{F}}
\def\GG{\mathcal{G}}\def\HH{\mathcal{H}}\def\II{\mathcal{I}}
\def\JJ{\mathcal{J}}\def\KK{\mathcal{K}}\def\LL{\mathcal{L}}
\def\MM{\mathcal{M}}\def\NN{\mathcal{N}}\def\OO{\mathcal{O}}
\def\PP{\mathcal{P}}\def\QQ{\mathcal{Q}}\def\RR{\mathcal{R}}
\def\SS{\mathcal{S}}\def\TT{\mathcal{T}}\def\UU{\mathcal{U}}
\def\VV{\mathcal{V}}\def\WW{\mathcal{W}}\def\XX{\mathcal{X}}
\def\YY{\mathcal{Y}}\def\ZZ{\mathcal{Z}}

\def\Ab{\mathbf{A}}\def\Bb{\mathbf{B}}\def\Cb{\mathbf{C}}
\def\Db{\mathbf{D}}\def\Eb{\mathbf{E}}\def\Fb{\mathbf{F}}
\def\Gb{\mathbf{G}}\def\Hb{\mathbf{H}}\def\Ib{\mathbf{I}}
\def\Jb{\mathbf{J}}\def\Kb{\mathbf{K}}\def\Lb{\mathbf{L}}
\def\Mb{\mathbf{M}}\def\Nb{\mathbf{N}}\def\Ob{\mathbf{O}}
\def\Pb{\mathbf{P}}\def\Qb{\mathbf{Q}}\def\Rb{\mathbf{R}}
\def\Sb{\mathbf{S}}\def\Tb{\mathbf{T}}\def\Ub{\mathbf{U}}
\def\Vb{\mathbf{V}}\def\Wb{\mathbf{W}}\def\Xb{\mathbf{X}}
\def\Yb{\mathbf{Y}}\def\Zb{\mathbf{Z}}

\def\ab{\mathbf{a}}\def\bb{\mathbf{b}}\def\cb{\mathbf{c}}
\def\db{\mathbf{d}}\def\eb{\mathbf{e}}\def\fb{\mathbf{f}}
\def\gb{\mathbf{g}}\def\hb{\mathbf{h}}\def\ib{\mathbf{i}}
\def\jb{\mathbf{j}}\def\kb{\mathbf{k}}\def\lb{\mathbf{l}}
\def\mb{\mathbf{m}}\def\nb{\mathbf{n}}\def\ob{\mathbf{o}}
\def\pb{\mathbf{p}}\def\qb{\mathbf{q}}\def\rb{\mathbf{r}}
\def\sb{\mathbf{s}}\def\tb{\mathbf{t}}\def\ub{\mathbf{u}}
\def\vb{\mathbf{v}}\def\wb{\mathbf{w}}\def\xb{\mathbf{x}}
\def\yb{\mathbf{y}}\def\zb{\mathbf{z}}

\def\Abb{\mathbb{A}}\def\Bbb{\mathbb{B}}\def\Cbb{\mathbb{C}}
\def\Dbb{\mathbb{D}}\def\Ebb{\mathbb{E}}\def\Fbb{\mathbb{F}}
\def\Gbb{\mathbb{G}}\def\Hbb{\mathbb{H}}\def\Ibb{\mathbb{I}}
\def\Jbb{\mathbb{J}}\def\Kbb{\mathbb{K}}\def\Lbb{\mathbb{L}}
\def\Mbb{\mathbb{M}}\def\Nbb{\mathbb{N}}\def\Obb{\mathbb{O}}
\def\Pbb{\mathbb{P}}\def\Qbb{\mathbb{Q}}\def\Rbb{\mathbb{R}}
\def\Sbb{\mathbb{S}}\def\Tbb{\mathbb{T}}\def\Ubb{\mathbb{U}}
\def\Vbb{\mathbb{V}}\def\Wbb{\mathbb{W}}\def\Xbb{\mathbb{X}}
\def\Ybb{\mathbb{Y}}\def\Zbb{\mathbb{Z}}

\def\Af{\mathfrak{A}}\def\Bf{\mathfrak{B}}\def\Cf{\mathfrak{C}}
\def\Df{\mathfrak{D}}\def\Ef{\mathfrak{E}}\def\Ff{\mathfrak{F}}
\def\Gf{\mathfrak{G}}\def\Hf{\mathfrak{H}}\def\If{\mathfrak{I}}
\def\Jf{\mathfrak{J}}\def\Kf{\mathfrak{K}}\def\Lf{\mathfrak{L}}
\def\Mf{\mathfrak{M}}\def\Nf{\mathfrak{N}}\def\Of{\mathfrak{O}}
\def\Pf{\mathfrak{P}}\def\Qf{\mathfrak{Q}}\def\Rf{\mathfrak{R}}
\def\Sf{\mathfrak{S}}\def\Tf{\mathfrak{T}}\def\Uf{\mathfrak{U}}
\def\Vf{\mathfrak{V}}\def\Wf{\mathfrak{W}}\def\Xf{\mathfrak{X}}
\def\Yf{\mathfrak{Y}}\def\Zf{\mathfrak{Z}}

\maketitle

\tableofcontents

\begin{abstract}
We exhibit an analogy between the problem of pushing forward measurable sets
under measure preserving maps and linear relaxations in combinatorial
optimization. We show how invariance of hyperfiniteness of graphings under
local isomorphism can be reformulated as an infinite version of a natural
combinatorial optimization problem, and how one can prove it by extending
well-known proof techniques (linear relaxation, greedy algorithm, linear
programming duality) from the finite case to the infinite.
\end{abstract}

\section{Introduction}\label{SEC:PRELIM}

\subsection{Pushing forward and pulling back}\label{SEC:PULL-PUSH}

Let $X$ and $Y$ be standard probability spaces, and let $\phi:~X\to
Y$ be a measure preserving map. By definition, we can {\it pull back}
a measurable subset $Y'\subseteq Y$: the set $\phi^{-1}(Y')$ is
measurable in $X$ and has the same measure as $Y'$. It is a lot more
troublesome to {\it push forward} a measurable subset $X'\subseteq
X$: the image $\phi(X')$ may not be measurable, and its measure may
certainly differ from the measure of $X'$.

On the other hand, if we have a measure $\mu$ on $X$ (possibly different from
the probability measure of $X$), then we can push it forward by the formula
$\mu^\phi(Y')=\mu(\phi^{-1}(Y'))$. So if we really have to push forward a
subset of $X$, we would like to ``encode'' it in a suitable sense by a measure,
and push forward this measure to $Y$. Then, of course, we still face the task
to ``distill'' a subset of $Y$ from this measure on $Y$. Measures can be pulled
back from $Y$ to $X$, at least if they have a density function $f$, using
$f^\phi(x)=f(\phi(x))$.

This vague description may sound familiar to a basic technique in combinatorial
optimization: linear relaxation. (It is also similar to ``fuzzy sets'', which
is a related setup.) The goal of this paper is to show that this analogy is in
fact much more relevant than it seems. We show how some important results in
graph limit theory (like invariance of hyperfiniteness under local
isomorphism), can be reformulated as infinite versions of natural combinatorial
optimization problems, and how one can prove them by extending well-known proof
techniques for the ``distillation'' of a set from a measure from the finite
case to the infinite. (Since the proof for the infinite case is described in
\cite{Hombook}, we only sketch it here.)

\subsection{Graphings}

For the rest of this paper, we fix a positive integer $D$, and all
graphs we consider are supposed to have maximum degree at most $D$.

A {\it graphing} is a Borel graph with bounded (finite) degree on a standard
probability space $(I,\AA,\lambda)$, satisfying the following
``measure-preserving'' condition for any two Borel subsets $A,B\subseteq I $:
\begin{equation}\label{EQ:UNIMOD}
\int\limits_A \deg_B(x)\,d\lambda(x) =\int\limits_B
\deg_A(x)\,d\lambda(x).
\end{equation}
Here $\deg_B(x)$ denotes the number of edges connecting $x\in I$ to
points of $B$. (It can be shown that this is a Borel function of
$x$.) Most of the time, we may assume that $I=[0,1]$ and $\lambda$ is
the Lebesgue measure.

Such a graphing defines a measure on Borel subsets of $I^2$: on
rectangles we define
\[
\eta(A\times B) = \int\limits_A \deg_B(x)\,d\lambda(x),
\]
which extends to Borel subsets in the standard way. We call this the
{\it edge measure} of the graphing. It is concentrated on the set of
edges, and it is symmetric in the sense that interchanging the two
coordinates does not change it.

For a graphing $\Gb$ and positive integer $r$, let us pick a random
element $x\in I$ according to $\lambda$, and consider the subgraph
$B_\Gb(x,r)$ induced by nodes at distance at most $r$ from $x$. This
gives us a probability distribution on {\it $r$-balls}: rooted graphs
with degrees at most $D$ and radius (maximum distance from the root)
at most $r$. Let $\rho_{\Gb,r}$ be the distribution of $B_\Gb(v,r)$.
We sometimes suppress the subscript $\Gb$ when the underlying
graphing is understood.

Two graphings $\Gb_1$ and $\Gb_2$ are called {\it locally
equivalent}, if $\rho_{\Gb_1,r}=\rho_{\Gb_2,r}$ for every $r\ge0$. To
characterize local equivalence, let us define a map $\phi:~V(\Gb_1)
\to V(\Gb_2)$ to be a {\it local isomorphism} from $\Gb_1$ to
$\Gb_2$, if it is measure preserving, and for every $x\in V(\Gb_1)$,
$\phi$ is an isomorphism between the connected component of $\Gb_1$
containing $x$ and the connected component of $\Gb_2$ containing
$\phi(x)$. It is easy to see that if there exists a local isomorphism
$\Gb_1\to\Gb_2$, then $\Gb_1$ and $\Gb_2$ are locally equivalent.

A local isomorphism may not be bijective, or even injective: it may
map different components of $\Gb_1$ on the same component of $\Gb_2$.
So it is not sufficient to characterize local equivalence. But making
it symmetric, we get a characterization \cite{Hombook}:

\begin{prop}\label{PROP:LOC-EQ}
Two graphings $\Gb_1$ and $\Gb_2$ are locally equivalent if and only
if there exists a third graphing $\Gb$ having local isomorphisms
$\Gb\to \Gb_1$ and $\Gb\to \Gb_2$.
\end{prop}

\subsection{Compact graphings}

Let $\Gb=(I,E,\lambda)$ be a graphing on a metric space $(I,d)$. For two points
$x,y\in I$ and integer $r\ge 0$, we define an {\it $r$-neighborhood
isomorphism} between $x$ and $y$ as an isomorphism $\phi:~B(x,r)\to B(y,r)$
such that $\phi(x)=y$. Such an isomorphism always exists for $r=0$. For an
$r$-neighborhood isomorphism $\phi$ between $x$ and $y$, we define
\begin{equation}\label{EQ:DIST}
d(\phi) = \max_{z\in B(x,r)} d(z,\phi(z)).
\end{equation}

We say that a graphing $\Gb=(I,d,E,\lambda)$ is {\it compact}, if $(I,d)$ is a
compact metric space, $E$ is a closed subset of $I\times I$, and for every
$\eps>0$ and integer $r\ge 0$ there is a $\delta>0$ such that for every pair
$x,y\in I$ with $d(x,y)\le\delta$ there exists an $r$-neighborhood isomorphism
$\phi$ between $x$ and $y$ with $d(\phi)\le \eps$.

We say that a graphing $\Gb=(I,\AA,E,\lambda)$ is a {\it full subgraphing} of a
graphing $\Gb'=(I',\AA',E',\lambda')$, if $\Gb$ is the union of connected
components of $\Gb'$, $I$ is a Borel subset of $I'$, $\AA=\AA'|_I$, and
$\lambda'(X)=\lambda(I\cap X)$ for every Borel subset $X\subseteq I'$. In
particular, $\lambda'(I)=\lambda(I)=1$, so $\lambda'(I'\setminus I)=0$.

The following was proved in \cite{Lov19}.

\begin{prop}\label{PROP:COMP}
Every graphing is a full subgraphing of a compact graphing.
\end{prop}

\subsection{Hyperfinite graphings}\label{SEC:HYPFIN}

There is an important special class of very ``slim'' graphings. For a
graphing $\Gb$, a set $T$ of edges will be called {\it
$k$-splitting}, if every connected component of $\Gb\setminus T$ has
at most $k$ nodes. We denote by $\sep_k(\Gb)$ the infimum of
$\eta(T)$, where $T$ is a $k$-splitting Borel set of edges. A
graphing $\Gb$ is {\it hyperfinite}, if $\sep_k(\Gb)\to0$ as
$k\to\infty$.

The following fact about hyperfiniteness is so basic and natural, that is quite
surprising that its proof is nontrivial.

\begin{theorem}\label{THM:HYP-W-ISO}
Let $\Gb_1$ and $\Gb_2$ be locally equivalent graphings. If $\Gb_1$
is hyperfinite, then so is $\Gb_2$.
\end{theorem}

This theorem was first proved in \cite{Hombook}. It is closely related to an
earlier theorem of Schramm \cite{Schramm1} about hyperfiniteness of locally
convergent graph sequences; we'll come back to this connection in the last
section. Independently, Elek \cite{Elek12} derived this result from a theorem
of Kaimanovich \cite{Kaim}. The proof in \cite{Hombook} yields an explicit
relationship between the values $\sep_k(\Gb_1)$ and $\sep_k(\Gb_2)$ (see
Corollary \ref{COR:HYP-W-ISO} below.).

Note that the stronger statement that $\sep_k(\Gb_1)=\sep_k(\Gb_2)$
is not true (Example 21.12 in \cite{Hombook}). An analogous
statement, however, will be true for the measure version (see Theorem
\ref{THM:SEP-STAR} below).

To illustrate the ``push forward -- pull back'' problem discussed in
the introduction, let us start to prove this theorem. Let $\eps>0$;
we want to prove that there is a $k\ge1$ such that
$\sep_k(\Gb_2)<\eps$; in other words, there is a $k$-splitting Borel
set $T_2\subseteq E(\Gb_2)$ such that $\eta_2(T_2)<\eps$ (where
$\eta_i$ denotes the edge measure of $\Gb_i$). By definition, there
is a $k\ge1$ and a $k$-splitting Borel set $T_1$ for $\Gb_1$ with
$\eta_1(T_1)<\eps$.

By Proposition \ref{PROP:LOC-EQ}, there is a third graphing $\Gb$
having local isomorphisms $\phi_1:~\Gb\to \Gb_1$ and $\phi_2:~\Gb\to
\Gb_2$. We can pull back the set $T_1$ to $\Gb$: the set
$T=\phi_1^{-1}(T_1)$ satisfies $\eta(T)=\eta_1(T_1)<\eps$, and (since
$\phi_1$ is a local isomorphism from $\Gb\setminus T$ to
$\Gb_1\setminus T_1$) the connected components of $\Gb\setminus T$
have no more than $k$ nodes. This shows that $\sep_k(\Gb)<\eps$.

To complete the proof, we would like to ``push forward'' the set $T$
to $\Gb_2$; but we have no control over what happens to its measure.
To get around this difficulty, we introduce a fractional version of
the $k$-splitting problem, which is defined in terms of a measure,
and thus it can be pushed forward in a manageable way. But we lose by
this, and the main step will be to estimate the loss.

\subsection{Convergent graph sequences}

Graphings were introduced (at least in this setting) as limit objects
of locally convergent sequences of bounded degree graphs \cite{AlLy,
Elek1}. Let us sketch this connection.

The probability distribution $\rho_{G,r}$ on $r$-balls can be defined
for finite graphs just as for graphings. We say that a sequence
$(G_n:~n=1,2,\dots)$ of finite graphs is {\it locally convergent}, if
for every $r\ge1$, the probability distributions $\rho_{G_n,r}$
converge (note that these distributions are defined on the same
finite set of $r$-balls, independently of $G_n$). This notion of
convergence was introduced by Benjamini and Schramm \cite{BS}.

We say that $G_n\to\Gb$ (where $\Gb$ is a graphing), if
$\rho_{G_n,r}\to\rho_{\Gb,r}$ for every $r\ge 1$. For every locally
convergent graph sequence $(G_n:~n=1,2,\dots)$ there is a graphing
$\Gb$ such that $G_n\to\Gb$. This fact can be derived from the work
of Benjamini and Schramm; it was stated explicitly in \cite{AlLy} and
\cite{Elek1}. It is clear that a convergent graph sequence determines
its limit up to local equivalence only.

\begin{remark}\label{REM:BS}
Benjamini and Schramm describe a limit object in the form of a
probability distribution on rooted countable graphs with degrees at
most $D$, with a certain ``unimodularity'' condition. This limit
object is unique. Graphings contain more information than what is
passed on to the limit. Among others, they can represent limits of
sequences that are convergent in a stronger sense called {\it
local-global convergence} \cite{HLSz}. But for us exactly the weaker
notion of convergence, and the uncertainty in the limit object it
introduces (local equivalence), is interesting.
\end{remark}

For a finite graph, the definition of hyperfiniteness makes no sense
(every finite graph is hyperfinite); we have to move to infinite
families of graphs. A family of finite graphs is {\it hyperfinite},
if for every $\eps>0$ there is an integer $k\ge1$ such that every
graph $G$ in the family satisfies $\sep_k(G)\le\eps$. Many important
families of graphs (with a fixed degree bound) are hyperfinite:
trees, planar graphs, and more generally, every non-trivial
minor-closed family \cite{BSchSh}. As a non-hyperfinite family, let
us mention any expander sequence.

The connection between hyperfinite graph families and hyperfinite
graphings is nice, and as it turns out, nontrivial:

\begin{theorem}\label{THM:HYPFIN}
Let $(G_n:~n=1,2,\dots)$ be a sequence of finite graphs with all
degrees bounded by $D$, locally converging to a graphing $\Gb$. Then
$\Gb$ is hyperfinite if and only if the family $\{G_n:~n=1,2,\dots\}$
is hyperfinite.
\end{theorem}

This theorem is due to Schramm \cite{Schramm1} (with a somewhat
sketchy proof). A complete proof based on other methods was described
by Benjamini, Schramm and Shapira \cite{BSchSh}. A third proof could
be based on the graph partitioning algorithm of Hassidim, Kelner,
Nguyen and Onak \cite{HKNO}. The proof in \cite{Hombook}, which is
based on Theorem \ref{THM:HYP-W-ISO} above, is perhaps closest to
Schramm's original method, although cast in a different form.

\section{Combinatorial version}

\subsection{Graph partitioning}

In this section, we discuss the finite version of the main tool in
the proof of Theorem \ref{THM:HYP-W-ISO}. Let $G=(V,E)$ be a finite
graph on $n$ nodes. The notion of $k$-splitting edge sets can be
defined for $G$ just as for graphings. We denote by $\sep_k(G)$ the
minimum of $|T|/n$, where $T$ is a $k$-splitting set of edges.

We formulate a relaxation of the problem of computing $\sep_k(G)$,
which can be expressed as a linear program. There are many ways to do
so (as usual); we choose one which is perhaps not the simplest, but
which will generalize to graphings easily.

Let $\RR_k=\RR_k(G)$ denote the set of subsets $A\subseteq V$ with
$1\le |A|\le k$ that induce a connected subgraph of $G$. An {\it
$\RR_k$-partition} of $V$ is a disjoint subfamily $\FF\subseteq\RR_k$
covering every node. For $A\subseteq V$, let $\partial A$ denote the
set of edges connecting $A$ to $V\setminus A$. We can express
$\sep_k(G)$ as
\begin{equation}\label{EQ:SEP-PART}
\sep_k(G)= \min_\FF \frac1{2n} \sum_{A\in\FF} |\partial A|,
\end{equation}
where $\FF$ ranges over all $\RR_k$-partitions. Indeed, if $T$ is
$k$-splitting, then the components of $G\setminus T$ form an
$\RR_k$-partition $\FF$ with $|T| = \frac12 \sum_{A\in\FF} |\partial
A|$. Conversely, for every $\RR_k$-partition $\FF$, the set
$T=\cup_{A\in\FF}\partial A$ is $k$-splitting, and $\sum_{A\in\FF}
|\partial A| = 2|T|$ (since every edge in $T$ is counted with two
sets $A$).

This suggests the following relaxation: A weighting $x:~\RR_k\to\R$
is a {\it fractional $\RR_k$-partition}, if
\begin{equation}\label{EQ:FR-COV}
x_A\ge0,\qquad \sum_{A\in\RR_k\atop A\ni v} x_A=1\qquad(\forall v\in V).
\end{equation}
We define
\begin{equation}\label{EQ:SEP-STAR}
\sep_k^*(G) = \min_x \frac1{2n} \sum_{A\in\RR_k} x_A |\partial A|,
\end{equation}
where $x$ ranges over all fractional $\RR_k$-partitions. The
indicator function of an $\RR_k$-partition is a fractional
$\RR_k$-partition, and hence
\begin{equation}\label{EQ:SIGMA-TRIV}
\sep^*_k(G)\le \sep_k(G).
\end{equation}
Equality does not hold in general: the triangle has $\sep_2(K_3)=2/3$
but $\sep_2^*(K_3)=\frac12$. But we have the following weak converse:

\begin{theorem}\label{THM:SEP}
For every finite graph $G$ with maximum degree $D$,
\[
\sep_k(G)\le \sep_k^*(G)\Bigl(2+\ln\frac{D}{2\sep_k^*(G)}\Bigr).
\]
\end{theorem}

This is the finite version of a result for graphings (Lemma 21.10 in
\cite{Hombook}). As we will see, it is crucial that the upper bound
depends on $\sep_k^*(G)$ and $D$ only, not on $k$ or $n$.

We give the proof of this theorem in the next section, in a more
general form. To get this more general form, we modify the
$k$-partition problem by looking for covering subgraph-families
rather than partitions. This does not change the value of $\sep_k$.
More exactly,
\begin{equation}\label{EQ:SEP-COV}
\sep_k(G)=\min_\FF\frac1{2n} \sum_{A\in\FF} |\partial A|,
\end{equation}
where $\FF$ ranges over families $\FF\subseteq\RR_k$ covering every
node. Indeed, allowing more families $\FF$ could only lower the
minimum. On the other hand, consider a covering family $\FF$
minimizing \eqref{EQ:SEP-COV}. If this consists of disjoint sets, we
are done. Suppose not, and let $A_1,A_2\in\FF$ such that $A_1\cap
A_2\not=\emptyset$. Let $e_1$ denote the number of edges between
$A_1\setminus A_2$ and $A_1\cap A_2$, and define $e_2$ similarly. We
may assume that $e_1\le e_2$. Replacing $A_1$ by $A_1'=A_1\setminus
A_2$, we still have a covering family, and since $|\partial A_1'|\le
|\partial A_1| + e_1 -e_2 \le |\partial A_2|$, we have another
optimizer in \eqref{EQ:SIGMA-DEF}. We can repeat this until all sets
in $\FF$ will be disjoint.

This suggests that we could use another linear relaxation. We define
a {\it fractional cover} as a vector $x\subseteq\R_+^{\RR_k}$ such
that $\sum_{A\ni v} x_A\ge 1$ for all $v\in V$. We define
\begin{equation}\label{EQ:SEP-STAR2}
\sep_k^{**}(G) = \min_x \frac1{2n} \sum_{A\in\RR_k} x_A |\partial A|,
\end{equation}
where $x$ ranges over all fractional covers. It is clear that
\[
\sep_k^{**}(G) \le \sep_k^*(G),
\]
and I don't know whether strict inequality can ever hold here.
Otherwise, there is not much difference in the behavior of these two
relaxations. In particular, Theorem \ref{THM:SEP} would remain valid
with $\sep_k^{**}$ instead of $\sep_k^*$. We are going to use
whichever is more convenient.

\subsection{Generalization to hypergraphs}

The theorem can be generalized to hypergraphs. Let $\HH$ be a
hypergraph on node set $V$, without isolated nodes, and let
$w:~\HH\to\R_+$ be an edge-weighting. An {\it edge-cover} is a subset
$\FF \subseteq \HH$, such that $\cup\FF=V$. A {\it fractional
edge-cover} is an assignment of weights $X:~\HH\to\R_+$ such that
\begin{equation}\label{EQ:FRAC-COV}
\sum_{A\ni v} x_A \ge 1\quad(\forall v\in V).
\end{equation}
Define
\begin{equation}\label{EQ:SIGMA-DEF}
\sigma=\sigma(\HH,w)= \min_\FF \frac{1}{n}\sum_{A \in \FF} w(A),
\end{equation}
where $\FF$ ranges over all edge-covers, and
\begin{equation}\label{SIGMA-ST-DEF}
\sigma^*=\sigma^*(\HH,w) = \min_x \frac{1}{n}\sum_{A\in\HH}  w(A) x_A,
\end{equation}
where $x$ ranges over all fractional edge-covers. Note that the
covering by singletons is in the competition, hence $\sigma^*\le 1$.

\begin{theorem}\label{THM:HYPERGR}
Let $\HH$ be a hypergraph on node set $V$ with $|V|=n$, such that $\{x\}\in\HH$
for each $x\in V$. Let $w:~\HH\to\R_+$ be an edge-weighting such that
$w(\{x\})\le1$ for every $x\in V$. Then
\[
\sigma^*\le \sigma \le \sigma^*\Bigl(2+ \ln\frac1{\sigma^*}\Bigr).
\]
\end{theorem}

\begin{proof}
The inequality $\sigma^*\le\sigma$ is trivial.

To prove the upper bound on $\sigma$, we use a version of the greedy
algorithm. For $i=1,2,\dots$, select edges $Y_1,Y_2,\dots,Y_m\in\HH$
so that $Y_i$ is a minimizer of
\[
\min_Y \frac{w(Y)}{|Y\setminus(Y_1\cup\dots\cup Y_{i-1})|}.
\]
(We don't consider edges $Y$ for which the denominator is zero.) We
stop when $\cup_iY_i=V$. Let $\FF=\{Y_1,\dots,Y_m\}$.

Set $y_i=|Y_i\setminus(Y_1\cup\dots\cup Y_{i-1})|$ and $w_i=w(Y_i)$.
We start with some simple inequalities. First, a partitioning into
singletons is a possibility in the definition of $\sigma$, and hence
\begin{equation}\label{EQ:SIGMA}
\sigma^*\le \sigma\le \frac1n\sum_{v\in V} w(\{v\})\le 1.
\end{equation}
When choosing $Y_i$ ($i<m$), any edge $A\in \HH$ with $A\setminus
(Y_1\cup\dots\cup Y_{i-1})\not=\emptyset$ was also available, but not
chosen. Hence
\begin{equation}\label{EQ:OPT}
\frac{w_i}{y_i}=\frac{w(Y_i)}{|Y_i\setminus(Y_1\cup\dots\cup Y_{i-1})|} \le
\frac{w(A)}{|A\setminus(Y_1\cup\dots\cup Y_{i-1})|}.
\end{equation}
In particular,
\[
\frac{w_i}{y_i}\le
\frac{w(Y_{i+1})}{|Y_{i+1}\setminus(Y_1\cup\dots\cup Y_{i-1})|} \le
\frac{w(Y_{i+1})}{|Y_{i+1}\setminus(Y_1\cup\dots\cup Y_i)|} = \frac{w_{i+1}}{y_{i+1}},
\]
and so
\begin{equation}\label{EQ:MON}
\frac{w_1}{y_1}\le \frac{w_2}{y_2}\le\dots\le\frac{w_m}{y_m}.
\end{equation}
When choosing $Y_m$, any singleton $v\in V\setminus(Y_1\cup\dots\cup
Y_{m-1})$ was available, showing that
\begin{equation}\label{EQ:ONE}
\frac{w_m}{y_m} \le \frac{w(\{v\})}{1}\le 1.
\end{equation}
We claim that
\begin{equation}\label{EQ:MAIN}
y_i+y_{i+1}+\dots+y_m \le \frac{y_i}{w_i} \sigma^*n.
\end{equation}
Indeed, if $x$ is the minimizer in the definition of $\sigma^*$, then
using \eqref{EQ:OPT},
\begin{align*}
y_i+ y_{i+1}+\dots+y_m &= |V\setminus(Y_1\cup\dots\cup Y_{i-1})|\\
&\le \sum_{v\in V\setminus(Y_1\cup\dots\cup Y_{i-1})}\ \sum_{A \ni v}
x_A= \sum_{A\in\HH} x_A\,|A\setminus(Y_1\cup\dots\cup Y_{i-1})|\\
&\le\sum_{A\in\HH} x_A\,\frac{y_i}{w_i} w(A) = \frac{y_i}{w_i}\sigma^*n.
\end{align*}

Now we turn to bounding $w(\FF)=w_1+\dots +w_m$. Let $1\le a\le m$ be
an integer such that
\[
\frac{w_{a-1}}{y_{a-1}}< \sigma^*\le \frac{w_a}{y_a}
\]
(possibly $a=0$ or $a=m+1$). Then
\[
w_1+\dots+w_{a-1}\le \sigma^*(y_1+\dots+y_{a-1}).
\]
For the remaining sum, multiply \eqref{EQ:MAIN} by $w_a/y_a$ for
$i=a$ and by $w_i/y_i-w_{i-1}/y_{i-1}$ for $i>a$, and sum the
inequalities. On the left side, the coefficient of $y_i$ will be
\[
\frac{w_a}{y_a}+\Bigl(\frac{w_{a+1}}{y_{a+1}}-\frac{w_a}{y_a}\Bigr)+\dots+
\Bigl(\frac{w_i}{y_i}-\frac{w_{i-1}}{y_{i-1}}\Bigr) = \frac{w_i}{y_i},
\]
and so on the left side we get
\[
\frac{w_a}{y_a}y_a+\dots+\frac{w_m}{y_m}y_m = w_a+\dots +w_m.
\]
On the right side, the coefficient of $\sigma^*n$ can be estimated as
follows, using that $w_a/y_a\ge \sigma^*$ and $w_m/y_m\le 1$:
\begin{align*}
\frac{w_a}{y_a}\,\frac{y_a}{w_a}& + \sum_{i=a+1}^m
\Bigl(\frac{w_i}{y_i}-\frac{w_{i-1}}{y_{i-1}}\Bigr)  \frac{y_i}{w_i}
\le 1+ \sum_{i=a+1}^m \int\limits_{w_{i-1}/y_{i-1}}^{w_i/y_i} \frac1t\,dt\\
&=1+ \int\limits_{w_a/y_a}^{w_m/y_m} \frac1t\,dt\le 1+
\int\limits_{\sigma^*}^1 \frac1t\,dt = 1+\ln \frac1{\sigma^*}.
\end{align*}
Thus
\[
w_1+\dots+w_m \le \sigma^*(y_1+\dots+y_{a-1}) + \Bigl(1+\ln \frac1{\sigma^*}\Bigr)\sigma^*n
\le \Bigl(2+\ln \frac1{\sigma^*}\Bigr)\sigma^*n.
\]
Dividing by $n$, we get the upper bound in the theorem.
\end{proof}

\begin{remark}\label{REM:RAND}
A simpler but non-algorithmic proof of Theorem \ref{THM:HYPERGR} can
be sketched as follows. Let again $x$ be the minimizer in the
definition of $\sigma^*$, let $t>0$, and form a family $\FF$ of edges
by selecting $A\in\HH$ with probability $p_A=\min(1,tx_A)$, and
adding all singletons that have not been covered. The probability
that a node $v$ is not covered by the randomly selected edges is at
most $e^{-t}$. This is trivial if $p_A=1$ for any of the hyperedges
containing $v$, and else it follows by the estimate
\[
\prod_{A\ni v} (1-p_A)\le \exp\Bigl(-\sum_{A\ni v} p_A\Bigr) =
\exp\Bigl(-t\sum_{A\ni v} x_A\Bigr) \le e^{-t}.
\]
So the expected number of edges on the boundaries of sets in $\FF$ is
at most
\[
\sum_{A\in\HH} t x_A|\partial A| + e^{-t}n = tn\sigma^*+e^{-t}n.
\]
This bound is minimized when $t=-\ln \sigma^*$, giving a very little
better bound than in the theorem.

The first proof we gave has two advantages: first, it is algorithmic,
and second (perhaps more importantly from our point of view), it
generalizes to the case of graphings.
\end{remark}

\begin{remark}\label{REM:ALL-1}
If all weights are $1$, and all edges $A\in\HH$ have cardinality
$|A|\le k$, then it is easy to see that $\sigma^*\ge 1/k$, and hence
\[
\sigma \le \sigma^*\Bigl(2+ \ln\frac1{\sigma^*}\Bigr) \le \sigma^*\Bigl(2+ \ln k\Bigr).
\]
With some care, we could reduce the first term from $2$ to $1$. This
is a well known inequality \cite{LL1}.
\end{remark}

\begin{remark}\label{REM:K-TRIV}
If we only know that $|A|\le k$ for all $A\in\HH$, but make no
assumption about the weights, we can still prove the (rather trivial)
inequality
\begin{equation}\label{EQ:SKS}
\sigma \le k\sigma^*.
\end{equation}
Indeed, for every node $v$, let $Z_v\in\HH$ be an edge containing $v$
with minimum weight. Then $\{Z_v:~v\in V\}$ is an edge-cover, and
hence
\[
\sigma\le \sum_v w(Z_v).
\]
On the other hand, let $(x_A:~A\in\HH)$ be an optimal fractional
edge-cover, then $\sum_{A\ni v} x_A\ge 1$, and hence
\begin{align*}
\sum_v w(Z_v) &\le \sum_v w(Z_v)\sum_{A\ni v}x_A = \sum_A x_A\sum_{v\in A} w(Z_v)\\
&\le\sum_A x_A\sum_{v\in A} w(A) = \sum_A x_A\,|A|\,w(A) \le k\sum_A x_A\,w(A) =k\sigma^*.
\end{align*}
\end{remark}

\medskip

\noindent{\bf Proof of Theorem \ref{THM:SEP}.} We consider the
hypergraph $(V,\RR_k(G))$ of connected subgraphs of size at most $k$,
with edge-weights $w(Y)=|\partial Y|/D$. We claim that
\begin{equation}\label{EQ:SEP-SIG}
\sep_k(G)= \frac{D}{2} \sigma(\HH,w).
\end{equation}
Indeed, for any family $\FF\subseteq\RR_k$, we have
\[
\frac12 \sum_{A\in\FF} |\partial A| = \frac{D}{2n} \sum_{A\in\FF} w(A),
\]
and the quantities on both sides of \eqref{EQ:SEP-SIG} are the minima
of the two sides of this last equation. The inequality
\begin{equation}\label{EQ:SIG-SEP-ST}
\sep_k^*(G) \ge \sep_k^{**}(G) = \frac{D}{2} \sigma^*(\HH,w)
\end{equation}
follows similarly. Using \eqref{EQ:SEP-SIG} and
\eqref{EQ:SIG-SEP-ST}, we can apply Theorem \ref{THM:HYPERGR}:
\[
\sep_k(G) = \frac{D}{2}\sigma \le
\frac{D}{2}\sigma^*\Bigl(2+ \ln\frac1{\sigma^*}\Bigr)
\le \sep_k^*(G)\Bigl(2+\ln\frac{D}{2\,\sep_k^*(G)}\Bigr).
\]

\subsection{Linear programming duality}

To motivate our results about duality for separation in graphings,
let us describe the (very standard) formulation of the dual of the
fractional $k$-separation problem.

The definition of $\sep_k^{**}(G)$ can be considered as a linear
program, with a variable $x_A$ for each set $A\in\RR_k$:
\begin{align}\label{EQ:SPLIT-D}
\text{minimize~~} &\quad \sum_{A\in\RR_k} \frac{|\partial A|}{2n} x_A,\nonumber\\
\text{subject to~}&\quad x_A\ge 0,\nonumber\\
                  &\quad\sum_{A\ni v} x_A \ge 1 \quad(\forall v\in V).
\end{align}
If $k$ is bounded, then this program has polynomial size, so together
with \eqref{EQ:SIGMA-TRIV} and Theorem \ref{THM:SEP}, we get a
polynomial time approximation algorithm for $\sep_k(G)$, with a
multiplicative error of $O(\ln(1/\sep^*_k(G)))$.

The dual program has a variable $y_v$ for each node $v$, and has the
following form (after some scaling):
\begin{align}\label{EQ:SPLIT-D2}
\text{maximize~~} &\quad \frac1{2n}\sum_{v\in V} y_v,\nonumber\\
\text{subject to~}&\quad y_v\ge 0,\nonumber\\
                  &\quad\sum_{v\in A} y_v \le |\partial A|
                  \quad(\forall A\in \RR_k).
\end{align}
This will give us a hint how to formulate ``dual'' in the graphing
setting.

\section{Fractional partitions in graphings}\label{FRAC-HYP}

\subsection{The space of connected subgraphs}

We consider a graphing $\Gb=(I,\AA,\lambda,E)$, and fix an integer
$k\ge1$. Just as for graphs, $\RR_k=\RR_k(\Gb)$ is the set of subsets
of $I$ inducing a connected subgraph with at most $k$ nodes. Every
singleton set belongs to $\RR_k$. Note that $\RR_k$ can be thought of
as a subset of $I \cup I^2\cup\dots\cup I^k$; it is easy to see that
this is a Borel subset. We can introduce a metric $d_k$ on $I\cup
I^2\cup\dots\cup I^k$: we use the $\ell_\infty$ metric on each $I^r$,
and distance $1$ between points in different sets $I^r$. It is clear
that $\RR_k$ is a closed subset of $I\cup I^2\cup\dots\cup I^k$, and
so it is a compact space.

Let $\tau$ be a finite measure on $\RR_k$. We define the {\it
marginal} of $\tau$ by
\[
\tau'(X) = \int\limits_{\RR^k} \frac{|X\cap Y|}{|Y|}\,d\tau(Y)\qquad (X\in\AA).
\]
In the case when $\tau$ is a probability measure, the probability
distribution $\tau'$ could be generated by selecting a set
$\Yb\in\RR_k$ according to $\tau$, and then a point $\yb\in\Yb$
uniformly. For every function $g:~I \to\R$ and every finite set
$Y\subseteq I $, define
\[
\overline{g}(Y)=\frac1{|Y|}\sum_{y\in Y} g(y).
\]
Then for every finite measure $\tau$ on $\RR_k$, and every integrable
function $g:~I \to\R$, we have
\begin{equation}\label{EQ:T-TP}
\int\limits_{\RR_k} \overline{g}\,d\tau = \int\limits_I g\,d\tau'.
\end{equation}
It is easy to see that this equation characterizes $\tau'$.

Mainly for technical purposes, we define a specific probability
distribution $\mu$ on the Borel sets of $\RR_k$ by selecting a random
point $x\in I$ according to $\lambda$, and then a subset $Y\ni x$
inducing a connected subgraph, uniformly among all such subsets.
Since there are only a finite number of such subgraphs, and at least
one, this makes sense. If $\RR_k(x)$ denotes the set of sets
$Y\in\RR_k$ containing $x$, then
\[
\mu(S)=\int\limits_I\frac{|S\cap\RR_k(x)|}{|\RR_k(x)|}.
\]
An obvious but important property of $\mu$ is that if $\mu(S)=0$ for
some $S\subseteq\RR_k$, then $\lambda(\cup S)=0$.

For a Borel subsets $T,U\subseteq\RR_k$, define
\begin{equation}\label{EQ:PHIDEF}
\Phi_T(x)= |T\cap\RR_k(x)|\qquad\text{and}\qquad
\mu_T(Z)=\mu(T\cap U).
\end{equation}
We need the following identity:
\begin{equation}\label{EQ:MU-PRIME}
\frac{d\mu_T'}{d\lambda}(x)= \Phi_T(x).
\end{equation}
(The left side is defined for almost all $x\in I$ only.) To prove
this, let $X$ be a Borel subset of $I$, and define
\[
f(x,y)=
  \begin{cases}
    \displaystyle \sum_{Y:\,x,y\in Y\in T} \frac1{|Y|} & \text{if $x\in X$}, \\
    0, & \text{otherwise}.
  \end{cases}
\]
Clearly $f(x,y)\not= 0$ implies that $x$ and $y$ belong to the same
component of $\Gb$ and their distance is at most $k-1$. Thus the sums
in the following equation are finite, and we can apply the Mass
Transport Principle (see e.g.\ Theorem 18.49 in \cite{Hombook}):
\[
\int\limits_I\sum_y f(x,y)\,d\lambda(x) = \int\limits_I \sum_x f(x,y)\,d\lambda(y).
\]
On the left, we get
\[
\int\limits_I\sum_y f(x,y)\,d\lambda(x) =
\int\limits_X \sum_y \sum_{Y:\,x,y\in Y\in T} \frac1{|Y|} \,d\lambda(x)
= \int\limits_X \sum_{Y:\,x\in Y\in T} \frac{|Y|}{|Y|} \,d\lambda(x)
=\int\limits_X \Phi_T\,d\lambda,
\]
while on the right,
\[
\int\limits_I \sum_x f(x,y)\,dy = \int\limits_I \sum_{x\in X}
\sum_{Y:\,x,y\in Y\in T} \frac1{|Y|} \,dy
=\int\limits_I \sum_{Y:\,y\in Y\in T} \frac{|X\cap Y|}{|Y|} \,dy
= \mu'_T(X).
\]
This proves \eqref{EQ:MU-PRIME}.

Let $g:~I\to\R$ be a bounded Borel function and define the integral
measure
\[
\lambda_g (X)=\int\limits_X g\,d\lambda.
\]
We need the identity
\begin{equation}\label{EQ:MU-G}
\int\limits_I\overline{g}\,d\mu_T = \int\limits_I \Phi_T(x)\,d\lambda_g.
\end{equation}
Indeed, using \eqref{EQ:MU-PRIME},
\[
\int\limits_{\RR_k}\overline{g}\,d\mu_T = \int\limits_I g\,d\mu'_T
= \int\limits_I g\,\frac{d\mu'_T}{d\lambda} \,d\lambda
=\int\limits_I g\Phi_T \,d\lambda = \int\limits_I \Phi_T(x)\,d\lambda_g.
\]

\subsection{Separation in graphings}

The definition $\sep_k^*$ can be extended to graphings with some
care. Here the probabilistic version of the definition is easier to
use. We say that $\tau$ is a {\it fractional $\RR_k$-partition}, if
its marginal is the uniform distribution on $I $, and $\tau$ is a
{\it fractional $\RR_k$-cover}, if its marginal majorizes the uniform
distribution.

For a finite set $Y\subseteq V(\Gb)$, we define its {\it edge
expansion} (or {\it isoperimetric number}) by
\begin{equation}\label{I-DEF}
\iso(Y)= \frac{|\partial Y|}{|Y|}
\end{equation}
We define the relaxed $k$-splitting value in terms of the ``expected
expansion'' of a fractional partition $\tau$; more exactly,
\begin{equation}\label{EQ:EXPEXP}
\sep^*_k(\Gb) = \frac{1}{2} \inf_\tau\, \E\iso(\Yb) =
\frac{1}{2} \inf_\tau \int\limits_{\RR_k} \iso(Y)\,d\tau(Y),
\end{equation}
where $\tau$ ranges over all fractional $\RR_k$-partitions, and $\Yb$
is a random subset from the distribution $\tau$.

Just as in the finite case, we would get the same value if we allowed
$\RR_k$-covers instead of $\RR_k$-partitions.

The main theorem relating $k$-splitting numbers and their fractional
versions extends to graphings:

\begin{theorem}\label{THM:SEP-INF}
For every graphing $\Gb$ with maximum degree $D$,
\[
\sep_k^*(\Gb) \le \sep_k(\Gb)\le
\sep_k^*(\Gb)\log\frac{8D}{\sep_k^*(\Gb)}.
\]
\end{theorem}

As an immediate corollary, we see that a graphing $\Gb$ is
hyperfinite if and only if $\sep^*_k(\Gb)\to0$ as $k\to\infty$.

The main difficulty in extending the previous proof to graphings is
that the greedy selection of the sets $Y_j$, as described in the
previous section, would last through an uncountable number of steps,
thus messing up all measurability conditions. Instead, we do the
selection in phases, where a (typically) uncountable number of
disjoint approximate minimizers are selected simultaneously. It turns
out that one can stop after a finite number of such phases. The prize
we have to pay is a small loss in the constant, as the logarithm here
is binary. (This construction uses some results from the theory of
Borel graphs.) We refer to \cite{Hombook} for details.

The following theorem shows that $\sep_k^*$ behaves better that
$\sep_k$ with respect to local equivalence.

\begin{theorem}\label{THM:SEP-STAR}
Let $\Gb_1$ and $\Gb_2$ be locally equivalent graphings. Then for
every $k\ge 1$, $\sep_k^*(\Gb_1)=\sep_k^*(\Gb_2)$.
\end{theorem}

Let us describe again the trivial half of the proof. The nontrivial
half will be given in Section \ref{SEC:APPDUAL}. It will be based
expressing $\sep^*_k(\Gb)$ in a dual form, in terms of the existence
of a function on $V(\Gb)$, and on the fact that (with some care)
functions can be pushed in both directions along a measure preserving
map.

By Proposition \ref{PROP:LOC-EQ}, it suffices to prove it in the case when
there is a local isomorphism $\phi:~\Gb_1\to\Gb_2$. To start with the easy
direction: there is a fractional $\RR_k(\Gb_1)$-partition $\tau_1$ such that
for a random set $Y\in\RR_k(\Gb_1)$ from this distribution, $\E(|\partial
Y|/|Y|)\le\sep_k(\Gb_1)$.

Since $\phi(Y)\in\RR_k(\Gb_2)$ for every $Y\in\RR_k(\Gb_1)$ by the
definition of local isomorphism, $\phi$ defines a map $\wh\phi:~
\RR_k(\Gb)\to\RR_k(\Gb_2)$. It is easy to see that this is a
measurable map, and hence it pushes forward the measure $\tau_1$ to a
probability measure $\tau_2$ on $\RR_k(\Gb_2)$. Furthermore, $\tau_2$
is a fractional partition (this follows since $\phi_2$ is measure
preserving), and
\[
\int\limits_{\RR_k(\Gb_2)}\iso(Y)\,d\tau_2(Y) =
\int\limits_{\RR_k(\Gb_1)}\iso(Y)\,d\tau_1(Y).
\]
Thus $\sep_k^*(\Gb_2) \le \sep_k^*(\Gb_1)$.

The previous two theorems imply a quantitative version of Theorem
\ref{THM:HYP-W-ISO}:

\begin{corollary}\label{COR:HYP-W-ISO}
Let $\Gb_1$ and $\Gb_2$ be locally equivalent graphings. Then
\[
\sep_k(\Gb_2)\le \sep_k(\Gb_1)\log\frac{8D}{\sep_k(\Gb_1)}.
\]
\end{corollary}

Indeed, we have $\sep^*_k(\Gb_2)=\sep^*_k(\Gb_1)\le \sep^*_k(\Gb_2)$,
and so by Theorem \ref{THM:SEP-INF},
\[
\sep_k(\Gb_2)\le \sep^*_k(\Gb_2)\log\frac{8D}{\sep_k^*(\Gb_2)} \le
\sep_k(\Gb_1)\log\frac{8D}{\sep_k(\Gb_1)}
\]
(using that the function $x\log(8D/x)$ is monotone increasing for
$x\le 1$).

Note that we have applied the nontrivial half of Theorem
\ref{THM:SEP-INF}, but (with some care) only the trivial half of
Theorem \ref{THM:SEP-STAR}.

\begin{remark}\label{REM:KSEP}
The inequality \eqref{EQ:SKS}, along with its proof, generalizes to
graphings with no difficulty: For every graphing $\Gb$,
\[
\sep_k(\Gb)\le k\sep_k^*(\Gb).
\]
However, this inequality would not be good enough to prove Theorem
\ref{THM:HYP-W-ISO}. Going through the proof, one could realize the
significance that the upper bound in Theorem \ref{THM:SEP-INF} does
not depend on $k$: we need the fact that if $\sep_k(\Gb_1)\to 0$ as
$k\to\infty$, then $\sep_k(\Gb_1)\log\frac{8D}{\sep_k(\Gb_1)}\to 0$,
and hence $\sep_k(\Gb_2)\to 0$.
\end{remark}

\section{Duality}

Our goal is to generalize the linear programming duality of the separation
problem to the graphing case.

\subsection{Dual formulation for graphings}

Recall that fractional partitions are probability distributions
$\tau$ on $\RR^k$ such that $\tau'=\lambda$. This can be expressed
explicitly by the conditions
\begin{equation}\label{EQ:T-L}
\int\limits_{\RR_k} \overline{g}\,d\tau = \int\limits_0^1 g\,d\lambda
\end{equation}
for all integrable functions $g:~I \to\R_+$. The relaxed separation number
$\sep_k^*$ is defined by
\begin{equation}\label{EQ:SEP-H}
\sep_k^*(\Gb)=\frac12\inf_\tau \int\limits_{\RR_k} \iso \,d\tau.
\end{equation}
A dual characterization (as a supremum instead of an infimum)
generalizing \eqref{EQ:SPLIT-D2}, is given in the following theorem.

\begin{theorem}\label{THM:DUAL}
For every graphing $\Gb=(I,\AA,\lambda,E)$, we have
\[
\sep_k^*(\Gb)=\sup \frac12 \int\limits_I g\,d\lambda,
\]
where $g$ ranges over all bounded Borel functions $g:~I\to \R$ such
that $\overline{g}\le\iso$ on $\RR_k$.
\end{theorem}

Note that the condition on $g$ can be written as $\sum_{y\in Y}
g(y)\le|\partial Y|$ for all $Y\in\RR_k$. The proof will show that if
the graphing is compact, then we can let $g$ range over continuous
functions only.

\begin{proof}
We may assume that the graphing is compact; else, we construct its
compactification as in Proposition \ref{PROP:COMP}, find the appropriate
functions $g$, and restrict them to the original points, where they are still
bounded Borel functions.

Let us denote the supremum on the right side by $\alpha_k$. First, we
show the ``easy direction''. For every measure $\tau$ on $\RR_k$ with
$\tau'=\lambda$, and every Borel function $g$ as in the theorem, we
have
\[
\int\limits_I g\,d\lambda = \int\limits_I g\,d\tau'
=\int\limits_{\RR_k} \overline{g} d\tau
\le \int\limits_{\RR_k} \iso d\tau,
\]
proving that $\sep_k^*(\Gb)\ge \alpha_k$.

To prove the opposite inequality, fix any $s<2\sep_k^*(\Gb)$. Let
$\CC(I)$ be the space of continuous functions on $I$, then for every
$g\in\CC(I)$, we have $\overline{g}\in \CC(\R_k)$. Define
\[
K_0=\Bigl\{g\in\CC(I):~\int\limits_I g\,d\lambda =s\Bigr\}
\quad\text{and}\quad
K=\bigl\{f\in\CC(\RR_k):~\exists g\in K_0,\, \overline{g}\le f\bigr\}.
\]
Both $K_0$ and $K$ are convex, and $K$ has nonempty interior. Both
$K_0$ and $K$ contain the constant function $s$ (defined on $I$ and
$\RR_k$, respectively).

Recalling the metric $d_k$ on $I\cup I^2\cup\dots\cup I^k$, we can
state the following.

\begin{claim}\label{CLAIM:H-CONT}
The function $\iso$ is continuous in the metric $d_k$.
\end{claim}

Indeed, fix a set $Y\in\RR_k$, and choose $N>k$ so that any two
points in $Y\cup\partial Y$ are at least $1/N$ apart. If
$d_k(Y,Z)<\min(1/N,\delta(N))$, then $|Y|=|Z|$, and the isomorphism
$\phi:~B(y,N)\to B(z,N)$ (where $y\in Y$, $z\in Z$ and $d(y,z)<1/N$)
maps $Y$ onto $Z$ and their neighborhoods onto each other, showing
that $\iso(Z)=\iso(Y)$.

The main step in the proof is the following claim.

\begin{claim}\label{CLAIM:H-K}
$\iso \in K$.
\end{claim}

Suppose not, then by the Hahn--Banach Theorem, there is a nonzero
continuous linear functional $\ell$ on $\CC(\RR_k)$ such that
$\ell(f)\ge\ell(\iso)$ for every $f\in K$. Since $\ell$ remains
bounded from below on $K$, it must be nonnegative on nonnegative
functions. We may normalize $\ell$ so that $\ell(1)=1$. For any $g\in
K_0$, the functional remains bounded from below on the linear variety
of functions of the form $\overline{f}$, where $f= s+t(g-s)$ for some
$t\in\R$. This implies that it must be constant on this variety and
so it must satisfy $\ell(\overline{g})=\ell(s)=s\ell(1)=s$. Thus we
have $\ell(\iso)\le \ell(s)=s$.

By the Riesz Representation Theorem, we can represent $\ell$ by a
measure $\tau$ on $\RR_k$ such that $\ell(f) = \int_{\RR_k} f\,
d\tau$ for every $f\in\CC(\RR_k)$. By the normalization $\ell(1)=1$,
$\tau$ is a probability measure. Furthermore,
\[
\ell(\overline{g}) =\int\limits_{\RR_k} \overline{g}\, d\tau =\int\limits_{I}g\, d\tau'=
s=\int\limits_{I} g\, d\lambda
\]
for every $g\in K_0$, which implies that $\tau'=\lambda$. But then
\[
\ell(\iso)=\int\limits_{\RR_k} \iso\, d\tau \ge 2\sep_k^*(\Gb)>s,
\]
a contradiction.

\smallskip

So $\iso\in K$, and hence there is a nonnegative function $g\in K_0$
such that $\overline{g}\le \iso$. So $2\alpha_k\ge s$. Since this
holds for every $s<2\sep_k^*(\Gb)$, we must have
$\alpha_k=\sep_k^*(\Gb)$.
\end{proof}

\subsection{Local isomorphism and hyperfiniteness}\label{SEC:APPDUAL}

We have two graphings $\Gb_t=(I_t,\AA_t,\lambda_t,E_t)$ ($t=1,2$) and a local
isomorphism $\phi:~\Gb_1\to\Gb_2$. We have seen that
$\sep_k^*(\Gb_2)\le\sep_k^*(\Gb_1)$. We want to conclude the proof of Theorem
\ref{THM:SEP-STAR} by proving that equality holds here.

To this end, let $0<s<\sep_k^*(\Gb_1)$, and recall the formula for
$\sep_k^*(\Gb_1)$ from Theorem \ref{THM:DUAL}: there is a bounded
Borel function $g_1:~I_1\to \R_+$ such that $\overline{g}_1\le\iso$
on $\RR_k$, and
\[
\frac12 \int\limits_{I_1} g_1\,d\lambda_1 > s.
\]
To ``push'' $g_1$ to $\Gb_2$, consider the measure $\lambda_{g_1}$
defined by
\[
\lambda_{g_1}(X)=\int\limits_X g_1\,d\lambda_1,
\]
and its ``push-forward'' defined by
$\gamma(X)=\lambda_{g_1}(\phi^{-1}(X))$. Since $g_1$ is bounded, we
have $\lambda_{g_1}\le K\lambda$ for some $K>0$, and hence
$\gamma(X)\le \lambda_1(\phi^{-1}(X)) = \lambda_2(X)$. It follows
that the Radon--Nikodym derivative
\[
g=\frac{d\gamma}{d\lambda_2}
\]
exists and it is bounded by $K$.

We claim that $g$ satisfies the conditions for $\Gb_2$ in Theorem
\ref{THM:DUAL}. The only nontrivial part is to show that
$\overline{g}\le\iso$.

Let us apply definitions \eqref{EQ:PHIDEF} to each $\Gb_t$, to get
functions $\Phi^t_T$ and measures $\mu^t$ and $\mu_T^t$. It is not
hard to verify (using that $\phi$ is a local isomorphism) that
\begin{equation}\label{EQ:PHI-INVAR}
\Phi^1_{\phi^{-1}(T)}(x)=\Phi^2_T(\phi(x))
\end{equation}
and for all $Y\in\RR_k(\Gb_1)$,
\[
\iso(\phi(Y))=\iso(Y).
\]

Let $T_2\subseteq\RR_k(\Gb_2)$ be any Borel set and let
$T_1=\phi^{-1}(T_2)$. By \eqref{EQ:MU-G} and \eqref{EQ:PHI-INVAR}, we
have
\begin{align*}
\int\limits_{T_2} \overline{g}\,d\mu_2 &= \int\limits_{I_2} \Phi^2_{T_2}(x)\,d\gamma(x) =
\int\limits_{I_1} \Phi^1_{T_1}(x)\,d\lambda_{g_1}(x)
=\int\limits_{T_1} \overline{g_1}\,d\mu_1\\
&\le \int\limits_{T_1} \iso\,d\mu_1= \int\limits_{T_2} \iso\,d\mu_2
\end{align*}
Since this holds for every $T_2$, it follows that
$\overline{g}\le\iso$ almost everywhere on $\RR_k(\Gb_2)$. Changing
the value of $g$ to $0$ on all point of all sets $Y\in\RR_k(\Gb_2)$
where $\overline{g}(Y)>\iso(Y)$, we get a function satisfying the
conditions of Theorem \ref{THM:DUAL}, and hence
\[
\sep_k^*(\Gb_2) \ge \int\limits_{I_2} g\,d\lambda_2 = \int\limits_{I_2} 1\,d\gamma
= \int\limits_{I_1} g_1d\lambda_1 > s.
\]
Since this holds for every $s<\sep_k^*(\Gb_1)$, it follows that
$\sep_k^*(\Gb_2)\ge \sep_k^*(\Gb_1)$.

\subsection{Splitting into finite parts}\label{APPDUAL2}

It is an alternative definition of hyperfinite graphs that one can delete a set
of edges with arbitrarily small measure so that the remaining graph has finite
components only. If the graph is not hyperfinite, we may be interested in
determining how small the measure of such an edge set can be. This question
motivates the following discussion.

Clearly $\sep_k$ is monotone decreasing in $k$, and hence $\lim_{k\to\infty}
\sep_k(\Gb)=\sep(\Gb)$ exists. It is not hard to see that $\sep(\Gb)$ is the
infimum of edge-measures of edge sets, whose deletion from $\Gb$ results in a
graph with every component finite.

We can define similarly $\lim_{k\to\infty}
\sep_k^*(\Gb)=\sep^*(\Gb)$, which is related to $\sep(\Gb)$ by
similar inequalities as in Theorem \ref{THM:SEP-INF}:
\begin{equation}\label{EQ:SEPG}
\sep^*(\Gb) \le \sep(\Gb)\le\sep^*(\Gb)\log\frac{8D}{\sep^*(\Gb)}.
\end{equation}
The graphing $\Gb$ is hyperfinite if and only if $\sep(\Gb)=0$ or,
equivalently, $\sep^*(\Gb)=0$. Duality gives the following nice
formula for $\sep^*(\Gb)$ of any graphing.

\begin{theorem}\label{THM:SEPG}
For every graphing,
\[
\sep^*(\Gb)=\sup_g \frac12\int_I g\,d\lambda,
\]
where $g$ ranges over all bounded Borel functions $g:~V(G)\to\R_+$
such that
\[
\sum_{y\in Y} g(y) \le |\partial Y|
\]
for every finite set $Y\subseteq V(\Gb)$ inducing a connected
subgraph.
\end{theorem}

\begin{proof}
The proof is similar to that in section \ref{SEC:APPDUAL}, but the
details are different. By Theorem \ref{THM:DUAL}, there exist bounded
Borel functions $g_k\ge 0$ $(k=1,2,\dots)$ such that
\begin{equation}\label{EQ:G-CONST}
\sum_{y\in Y} g_k(y) \le |\partial Y|\qquad (\forall Y\in\RR_k),
\end{equation}
and
\begin{equation}\label{EQ:G-CONST2}
\int\limits_I g_k\,d\lambda\ge \sep_k^*(\Gb)-\frac1k.
\end{equation}
Note that $g_k(x) = \overline{g}_k(\{x\}) \le |\partial\{x\}| \le D$,
so these functions remain uniformly bounded. It follows by Alaoglu's
Theorem that they have a weak$^*$ limit $g$ in $L_\infty(J,\lambda)$;
this limit satisfies
\begin{equation}\label{EQ:G-WEAK}
\int\limits_I g_k\,f\,d\lambda \to \int\limits_I g\,f\,d\lambda
\quad(k\to\infty)
\end{equation}
for every $f\in L_1(I,\lambda)$. In particular, it follows that $0\le
g\le D$ (almost everywhere, but we may change $g$ on a zero set, so
that this holds everywhere), and
\[
\|g\|_1 = \int\limits_I g\,d\lambda \ge \sep^*(\Gb).
\]
We claim that $\overline{g}\le\iso$ holds almost everywhere on
$\RR_1\cup\RR_2\cup\dots$. It suffices to prove this on $\RR_k$ for a
fixed $k$.

Let $T\subseteq\RR_k$ be any Borel set. The Radon-Nikodym derivative
$\Phi_T=d\mu_T'/d\lambda$ exists and is bounded by
\eqref{EQ:MU-PRIME}. By \eqref{EQ:MU-G},
\[
\int\limits_{T} \overline{g}\,d\mu
=\int\limits_I g\,d\mu_T' =\int\limits_I g\,\Phi_T\,d\lambda.
\]
Similarly,
\[
\int\limits_{T} \overline{g}_k\,d\mu =\int\limits_I g_k\,\Phi_T\,d\lambda.
\]
Since $\overline{g}_k\le\iso$, we have
\[
\int\limits_I g_k\,f\,d\lambda \le \int\limits_{T} \iso\,d\mu,
\]
and hence by weak convergence,
\[
\int\limits_{T} \overline{g}\,d\mu \le \int\limits_{T} \iso\,d\mu.
\]
This holds for every $T$, which implies that $\overline{g}\le\iso$
almost everywhere. Changing $g$ to $0$ on all points of $I$ in sets
violating $\overline{g}\le\iso$, we get a function as in the theorem.
\end{proof}

\begin{corollary}\label{COR:HYPFIN}
A graphing $\Gb$ is not hyperfinite if and only if there exists a
bounded Borel function $g:~V(G)\to\R_+$, not almost everywhere zero,
such that
\[
\sum_{y\in Y} g(y) \le |\partial Y|
\]
for every finite set $Y\subseteq V(\Gb)$ inducing a connected
subgraph.
\end{corollary}

\section{Concluding questions and remarks}

{\bf 1.} Is the logarithmic factor needed in Theorems \ref{THM:SEP}
and \ref{THM:SEP-INF}, or in inequality \eqref{EQ:SEPG}? Could a
constant factor be enough?

For Theorem \ref{THM:HYPERGR}, in the general setting of hypergraphs,
it is easy show that a logarithmic factor is needed. For example, let
$\HH$ be the hypergraph whose vertices are all $p$-element subsets of
an $2p$-element set $S$, and edges are subfamilies containing a given
element of $S$. Let all hyperedges have weight $1$. Then any $p+1$
hyperedges cover $V(\HH)$, but no $p$ of them does, and so
\[
\sigma(\HH,w)= \frac{p+1}{\binom{2p}{p}}.
\]
On the other hand, $x_A = 1/p$ ($A\in E(\HH)$) defines a fractional
cover, which is easily seen to be optimal, and hence
\[
\sigma^*(\HH,w)= \frac{2}{\binom{2p}{p}}.
\]
So we see that
\[
\sigma(\HH,w)\sim \frac14 \sigma^*(\HH,w)\log
\frac1{\sigma^*(\HH,w)}
\]

\medskip

\noindent{\bf 2.} Is $\sep_k^{*}(G)=\sep_k^{**}(G)$ for every graph $G$? In
other words, can the simple manipulation described after the statement of
Theorem \ref{THM:SEP} be extended from covers to fractional covers? If not, how
far apart can these two parameters be?

\medskip

\noindent{\bf 3.} Can we improve the bound in Theorem \ref{THM:SEP}
by allowing larger components? Perhaps it is true that
\[
\lim_{m\to\infty} \sep_m(\Gb)\le \sep^*_k(\Gb)
\]
for every graphing $\Gb$ and integer $k\ge 1$. This would imply that
$\sep^{*}(\Gb)=\sep(\Gb)$ for every graphing $\Gb$.

\medskip

\noindent{\bf 4.} A central goal in the theory of very large graphs (and in the
theory of their limits) is to encode them, approximately, in a compact way. In
the case of dense graphs and their limit graphons, the Regularity Lemma of
Szemer\'edi \cite{Szem1,Szem2} provides such an encoding: Given a graph $G$ and
a positive integer $k$, we can construct its regularity partition
$\{V_1,\dots,V_k\}$ into $k$ (almost equal) classes, and record the edge
densities $a_{ij}$ between any two classes $V_i$ and $V_j$. This $k\times k$
matrix $(a_{ij})_{i,j=1}^k$ contains enough information to determine the
approximate value of the subgraph densities of $G$, and much more. (The
dependence between $k$ and the error we commit depends on which version of the
Regularity Lemma we talk about: Weak (Frieze and Kannan \cite{FK}), Original
(Szemer\'edi \cite{Szem2}), or Strong (Alon, Fischer, Krivelevich and Szegedy
\cite{AFKS}.)

Furthermore, we can reconstruct an imitation of the original graph $G$ by
taking $k$ sets $U_1,\dots,U_k$ of the same size, and putting in a random
bipartite graphs of the given density $a_{ij}$ between any two $U_i$, $U_j$ of
these sets. We don't have to know the size of the original graph for this
construction. This allows us to ``scale up'' or ``scale down'' the graph with
little change in its basic properties (subgraph densities, spectra, maximum
cut, etc.).

The applications of this Lemma in extremal graph theory, computer science,
graph limit theory and elsewhere are so numerous that it would be impossible
even to give a glimpse of them.

In the case of bounded degree graphs, or graphings, no similarly far-reaching
``Regularity Lemma'' has been found so far; extensions from dense graphs to
sparser graphs (see e.g.~\cite{Sco}) are nontrivial, but unfortunately they
become weaker and weaker as the edge-density decreases. One class for which a
compact encoding can be given, serving similar goals, are hyperfinite graphs.

Let $G$ be a graph such that $\sep_k(G)\le\eps$. Then we can delete an $\eps$
fraction of the edges so that every connected component of the remaining graph
has at most $k$ nodes. Specifying the fraction $a_H$ of connected components
isomorphic to a given graph $H$ creates a bonded number of data, which contain
enough information about the graph to determine, approximately, the
distribution of $r$-neighborhoods (for values of $r$ bounded by some function
of $\eps$), and again, much more. We can construct an imitation of the original
graph $G$ by taking the disjoint union of connected graphs with at most $k$
nodes, where the number of copies each of these graphs $H$ is proportional to
$a_H$.

While hyperfiniteness seems restrictive, many of the most important graph
classes (for example, planar graphs with bounded degree) are hyperfinite. It
would be interesting to test to what degree are real-life graphs hyperfinite
(for a moderately small $\eps$ and a $k$ that is small for the given graph
size). Perhaps the algorithms described in this paper will be useful for this.
The work of Kaimanovich \cite{Kaim} shows that if a graph is not sufficiently
hyperfinite, then it contains an expander in some sense. Since expander graphs
are notoriously difficult to construct, this indicates that hyperfinite graphs
may be more common than one would think.

It would be interesting to extend the notion of hyperfiniteness in the
following direction. We would like to delete a set of $T$ edges to get a graph
with all connected components having at most $k$ nodes; but now $T$ does not
have to small, just have small complexity in some sense. As a toy example, we
can construct a graph $G'$ by adding a random perfect matching to a graph $G$
with components having at most $k$ nodes. We can determine the neighborhood
statistics of $G'$ from the frequencies $a_H$ in $G$. Can we reconstruct $G$
from $G'$?

\end{document}